\documentclass{amsart}

\usepackage{amsthm,amsmath,amsthm,amssymb,mathrsfs,graphicx,mathtools,relsize,hyperref,cleveref,enumerate,nameref}
\usepackage[utf8]{inputenc}
\usepackage[all]{xy}

\theoremstyle{theorem}
 \newtheorem{thm}{Theorem}[section]
 \newtheorem{prop}[thm]{Proposition}
 \newtheorem{lem}[thm]{Lemma}
 \newtheorem{cor}[thm]{Corollary}

 \newtheorem*{thmA}{Theorem A}
\theoremstyle{definition}

\theoremstyle{remark}
 
 \numberwithin{equation}{section}
 
\renewcommand{\le}{\leqslant}

\renewcommand{\setminus}{\smallsetminus}
\def\supp{\text{\rm supp}}

\def\acts{\curvearrowright}
\newcommand{\norm}[1]{\left\lVert#1\right\rVert}

\title[The affine group of a local field is Hermitian]{The affine group of a local field is Hermitian}
\date{\today}

\subjclass[2020]{22D15, 43A20, 43A45, 46H05}
\keywords{Hermitian Banach $*$-algebra, affine group, local field, number field.}
\thanks{The author acknowledges support from the FWO and F.R.S.-FNRS under the Excellence of Science (EOS) program (project ID 40007542).}

\author[Max Carter]{Max Carter} 
\address{Institut de recherche en mathématique et physique \\
Université Catholique de Louvain \\ 
Chemin du Cyclotron 2 \\
boîte L7.01.02 \\
1348 Louvain-la-Neuve \\
Belgique.}
\email{max.carter@uclouvain.be}

\begin{document}

\begin{abstract}
The question of whether the group $\mathbb{Q}_p \rtimes \mathbb{Q}_p^*$ is Hermitian has been stated as an open question in multiple sources in the literature, even as recently as a paper by R.\! Palma published in 2015. In this note we confirm that this group is Hermitian by proving the following more general theorem: given any local field $\mathbb{K}$, the affine group $\mathbb{K} \rtimes \mathbb{K}^*$ is a Hermitian group. The proof is a consequence of results about Hermitian Banach $*$-algebras from the 1970's. In the case that $\mathbb{K}$ is a non-archimedean local field, this result produces examples of totally disconnected locally compact Hermitian groups with exponential growth, and these are the first examples of groups satisfying these properties. This answers a second question of Palma about the existence of such groups. 
\end{abstract}

\maketitle


\section{Introduction}

A locally compact group $G$, or equivalently the Banach $*$-algebra $L^1(G)$, is called \textit{Hermitian} if all self-adjoint $f \in L^1(G)$ have real spectrum. It is a classical question in harmonic analysis and Banach algebra theory to determine which locally compact groups are Hermitian \cite{Lep84}. This question has been studied extensively since the mid-1900's, particularly in the cases of connected Lie groups and discrete groups. See \cite[Section 12.6.22]{Pal01} and \cite{SW20} for more details on the current state of knowledge. 

It is well known that the real $ax+b$ group, $\mathbb{R} \rtimes \mathbb{R}_{>0}$, is a Hermitian group. This fact is typically attributed to the following German paper of Leptin \cite[Satz 6]{Lep76c}, but also follows from a later classification of solvable connected simply-connected Hermitian Lie groups of dimension $\le 6$ \cite{LP79,Pog79}. On the other hand, it is stated as an open question in, for example, \cite[Page 1490]{Pal01} and \cite[Section 3.6]{Pal15}, to determine whether the group $\mathbb{Q}_p \rtimes \mathbb{Q}_p^*$ is Hermitian. In this note, we answer this question by proving the following more general theorem.

\begin{thmA}
Let $\mathbb{K}$ be a local field. The group $\mathbb{K} \rtimes \mathbb{K}^*$ is Hermitian.
\end{thmA}

When $\mathbb{K}$ is a non-achimedean local field, the group $\mathbb{K} \rtimes \mathbb{K}^*$ is a non-discrete totally disconnected locally compact group. Furthermore, this group is non-unimodular and hence has exponential growth. This theorem thus produces examples of (non-discrete) totally disconnected locally compact Hermitian groups with exponential growth, answering a question of Palma about the existence of such groups, see \cite[Question 2, page 266]{Pal13} and \cite[Section 3.6]{Pal15}. In contrast, a finitely generated solvable discrete group is Hermitian if and only if it has polynomial growth \cite[Page 277]{Lep76}. It should also be noted that, as far as the author is aware, it is still an open question as to whether there exists a discrete Hermitian group with exponential growth, see \cite{Pal13,Pal15} again. 

The proof of Theorem A is a consequence of Satz 5 in Leptin's original article \cite{Lep76c} (we will refer to this theorem as \textit{Leptin's theorem} from now on). The point of this note is to translate the critical parts of Leptin's paper \cite{Lep76c} into English, and provide the few extra details required to verify that $\mathbb{K} \rtimes \mathbb{K}^*$ is a Hermitian group whenever $\mathbb{K}$ is a local field. 

If we now instead take $\mathbb{K}$ to be a number field equipped with the discrete topology, then it is also well known that the discrete group $\mathbb{K} \rtimes \mathbb{K}^*$ is \textit{not} a Hermitian group since it contains free subsemigroups on two or more generators \cite{Jen70}. As an example to conclude the article, we show that in the case of the group $\mathbb{K} \rtimes \mathbb{K}^*$ when $\mathbb{K}$ is a number field, only one of the assumptions of Leptin's theorem is not satisfied. In particular, the assumptions in Leptin's theorem are strictly necessary to get a Hermitian Banach $*$-algebra, which has implications on other questions the author is currently working on.


\section{Preliminaries on generalised $L^1$-algebras and Fourier algebras}\label{sec:2}

In this section we define some notation and terminology concerning generalised $L^1$-algebras and Fourier algebras. We refer the reader to \cite{Lep65,Lep67,Lep76} for more details on generalised $L^1$-algebras. Throughout this section $G$ is a locally compact group and all integration is performed against a fixed left Haar measure on $G$.

Let $A$ be a Banach $*$-algebra and $G$ a locally compact group acting strongly continuously on $A$ by isometric $*$-automorphisms. Given $a \in A$ and $x \in G$, let $a^x$ denoted the image of $a$ under $x$ with respect to this action. The \textit{generalised $L^1$-algebra} associated to $G$ and $A$ is the space $L^1(G,A)$ of measurable $A$-valued functions on $G$ which are integrable against the Haar measure. It forms a Banach $*$-algebra with the product
\begin{displaymath} f \star g(x) := \int_G f(xy)^{y^{-1}}g(y^{-1}) \: dy \end{displaymath}
and involution
\begin{displaymath} f^*(x) := \Delta(x)^{-1} (f(x^{-1})^{x})^* \end{displaymath} 
where $\Delta$ is the modular function on $G$. The norm on $L^1(G,A)$ is the obvious one:
\begin{displaymath} \norm{f}_{L^1(G,A)} := \int_G \norm{f(x)}_{A} \: dx. \end{displaymath}

Now let's further assume that $G$ is abelian. Then, the \textit{Fourier algebra} of $G$, denoted $A(G)$, is the image of $L^1(\widehat{G})$ under the Fourier transform on $\widehat{G}$. The algebra $A(G)$ is a $*$-subalgebra of $C_0(G)$, and it is a Banach $*$-algebra when equipped with a certain norm which we denote by $\norm{\:\cdot\:}_{A(G)}$. We will not go further into the details of Fourier algebras here, but the reader can consult the book \cite{KL18} for more details if desired. 

We now note the following result which will be critical in this article.

\begin{prop}\label{prop:genl1}
Let $G := N \rtimes H$ be a semi-direct product of locally compact abelian groups with $N$ normal in $G$. Then, we have isometric $*$-isomorphisms:
\begin{displaymath}  L^1(G) \cong L^1(H, L^1(N)) \cong L^1(H,A(\widehat{N})). \end{displaymath}
\end{prop}

\begin{proof}
The first isomorphism follows from standard results on generalised $L^1$-algebras, see \cite[Satz 10]{Lep65}. The second isomorphism follows from the fact that $L^1(N)$ is isometrically $*$-isomorphic to $A(\widehat{N})$ for any locally compact abelian group $N$, see \cite[Remark 2.4.5]{KL18}. 
\end{proof}


\section{Translations from Leptin's article}\label{sec:3}

In this section we translate the necessary parts of Leptin's article \cite{Lep76c} to English and set some further notation.

Let $G$ be a locally compact group and $X$ a locally compact $G$-space. Let $A \subseteq C_0(X)$ be a $*$-subalgebra satisfying the following four properties: 
\begin{enumerate}[(i)]
   \item $A$ is a Banach $*$-algebra with respect to some norm denoted $\norm{f}_A$ (possibly different to the uniform norm on $C_0(X)$);\
   \item $A$ is left translation invariant i.e.\ for $f \in A$ and $g \in G$, $f^g(x) := f(gx)$ is in $A$ and $\norm{f^g}_A = \norm{f}_A$;\
   \item For every $f \in A$, the map $g \mapsto f^g$ is continuous from $G$ to $A$;
   \item $A$ is a regular Banach algebra and the functions of compact support in $A$ are dense in $A$.
\end{enumerate}

Now, for a closed subset $T \subseteq X$, define
\begin{displaymath} k(T):= \{ f \in A : f(x)=0 \: \forall x \in T \} \end{displaymath}
and
\begin{displaymath}  j(T) := \overline{\{ f \in A \cap C_c(X) : \supp(f) \cap T = \emptyset\}}. \end{displaymath}
The sets $k(T)$ and $j(T)$ are respectively the largest and smallest ideals of $A$ with zero set $T$, in particular, $j(T) \subseteq k(T)$. The set $T$ is called a \textit{set of synthesis} or \textit{Wiener} (with respect to $A$) if $j(T) = k(T)$. Then, Leptin proves the following theorem in \cite{Lep76c}, which we refer to as \textit{Leptin's theorem}.

\begin{thm}\cite[Satz 5]{Lep76c}\label{thm:lep}
Suppose that $G$ is a $\sigma$-compact locally compact abelian group and $X$ a locally compact $G$-space. Let $A$ be a $*$-subalgebra of $C_0(X)$ satisfying properties (i) through to (iv) as above. If the action of $G$ on $X$ has finitely many orbits, and the closed $G$-invariant subsets of $X$ are sets of synthesis, then the Banach $*$-algebra $L^1(G,A)$ is Hermitian. 
\end{thm}

The following result is then a direct corollary of the theorem. It is not stated explicitly in Leptin's article \cite{Lep76c}, but it is alluded to in the discussion of the article.

\begin{cor}\label{cor:main}
Let $G := N \rtimes H$ be a semi-direct product of locally compact abelian groups with $N$ normal in $G$ and $H$ $\sigma$-compact. If the action $H \acts \widehat{N}$ has finitely many orbits, and the closed $H$-invariant subsets of $\widehat{N}$ are sets of synthesis with respect to $A(\widehat{N})$, then $G$ is Hermitian.
\end{cor}

\begin{proof}
By Proposition \ref{prop:genl1}, we have an isomorphism $L^1(G) \cong L^1(H,A(\widehat{N}))$, so it suffices to check that $L^1(H,A(\widehat{N}))$ is Hermitian. By standard facts on Fourier algebras (see \cite[Section 2]{KL18}), the Fourier algebra $A(\widehat{N})$ satisfies properties (i) through to (iv) as listed at the start of this section, where we take $X = \widehat{N}$. The corollary then follows by Theorem \ref{thm:lep}. 
\end{proof}


\section{Proof of Theorem A}

We now prove Theorem A from the introduction. Throughout this section we assume that $\mathbb{K}$ is a local field. To prove Theorem A, we just need to show that the action $\mathbb{K}^* \acts \widehat{\mathbb{K}}$ satisfies the assumptions as in Corollary \ref{cor:main}, where $H=\mathbb{K}^*$ and $N=\mathbb{K}$.

To do this, we first remind the reader of the following description of the Pontryagin dual $\widehat{\mathbb{K}}$.

\begin{lem}\cite[Proposition 19, page 234]{Bou19}\label{lem:Bourbaki}
Let $\mathbb{K}$ be a local field, viewed as an abelian group under addition, and let $\chi$ be a non-trivial unitary character of $\mathbb{K}$. Given $y \in \mathbb{K}$, define $\chi_y(x) := \chi(yx)$ for $x \in \mathbb{K}$. Then, the map
\begin{displaymath} \mathbb{K} \rightarrow \widehat{\mathbb{K}}, y \mapsto \chi_y \end{displaymath}
is an isomorphism of locally compact abelian groups.
\end{lem}

One then uses this lemma to compute that for $x,y \in \mathbb{K}$ and $z \in \mathbb{K}^*$,
\begin{displaymath} z \cdot \chi_y(x) = \chi_y(zx) = \chi(yzx) = \chi_{zy}(x).  \end{displaymath}

In particular, it follows that under the identification $\mathbb{K} \cong \widehat{\mathbb{K}}$ as in Lemma \ref{lem:Bourbaki}, the action of $\mathbb{K}^*$ on $\widehat{\mathbb{K}}$ is identified with the canonical multiplicative action of $\mathbb{K}^*$ on $\mathbb{K}$. Thus, we will work with the action $\mathbb{K}^* \acts \mathbb{K}$ instead for the remainder of the proof.

Then, one computes easily that there are two orbits of the action $\mathbb{K}^* \acts \mathbb{K}$: $\{0\}$ and $\mathbb{K} \setminus \{0\}$. Furthermore, the closed $\mathbb{K}^*$-invariant subsets of $\mathbb{K}$ are $\{0\}$ and $\mathbb{K}$. To complete the proof, we just need to show that the sets $\{0\}$ and $\mathbb{K}$ are sets of synthesis with respect to the algebra $A(\mathbb{K})$.

Clearly $j(\mathbb{K}) = k(\mathbb{K})$ since both $j(\mathbb{K})$ and $k(\mathbb{K})$ contain only the zero function. To complete the proof we just need to show that $j(\{0\}) \supseteq k(\{0\})$ (and hence $j(\{0\}) = k(\{0\})$). Suppose that $f \in k(\{0\}) \subseteq A(\mathbb{K})$. We must show that $f \in j(\{0\})$. For each $n \in \mathbb{N}$, by \cite[Lemma 2.3.7]{KL18}, there exists a function $h_n \in A(\mathbb{K}) \cap C_c(G)$ satisfying the following:
\begin{enumerate}[(i)]
   \item $h_n$ vanishes in a neighbourhood of $0 \in \mathbb{K}$;\
   \item $\norm{h_n -f}_{A(G)} \le 1/n$.
\end{enumerate}
In particular, it follows that $h_n \in j(\{0\})$ for each $n \in \mathbb{N}$ and $h_n$ converges to $f$ in $A(\mathbb{K})$ as $n \rightarrow \infty$. Thus $f \in j(\{0\})$ by definition of $j(\{0\})$. This completes the proof.


\section{The case of a number field}

For this section, we now let $\mathbb{K}$ be a number field. In particular, this field has the discrete topology. As mentioned in the introduction, since $\mathbb{K} \rtimes \mathbb{K}^*$ is discrete and contains free subsemigroups on two or more generators, it follows by the results of \cite{Jen70} that this semi-direct product is not a Hermitian group. 

We now show that the action $\mathbb{K}^* \acts \widehat{\mathbb{K}}$ does not have finitely many orbits, but the closed $\mathbb{K}^*$-invariant subsets of $\widehat{\mathbb{K}}$ are sets of synthesis.  

Let $\mathbb{A}_\mathbb{K}$ denote the ring of adeles over the number field $\mathbb{K}$. It is well known that we have an isomorphism $\widehat{\mathbb{K}} \cong \mathbb{A}_\mathbb{K}/\mathbb{K}$ \cite[Proposition 7.15]{RV99}. Furthermore, $\widehat{\mathbb{K}}$ is uncountable since it is a quotient of an uncountable group by a countable group, and $\widehat{\mathbb{K}}$ is compact since $\mathbb{K}$ is discrete \cite[Theorem 23.17]{HR79}.

It is a known result in harmonic analysis that every closed subset of a compact abelian group is a set of synthesis \cite[Example 39.10(b)]{HR70}. In particular, the closed $\mathbb{K}^*$-invariant subsets of $\widehat{\mathbb{K}}$ are sets of synthesis. However, since $\mathbb{K}^*$ is a countable group, the orbits of the action $\mathbb{K}^* \acts \widehat{\mathbb{K}}$ are countable, hence there must be uncountably many orbits for this action since $\widehat{\mathbb{K}} \cong \mathbb{A}_\mathbb{K}/\mathbb{K}$ is uncountable.

This shows that one cannot remove the assumption of having finitely many orbits from Corollary \ref{cor:main} and still be guaranteed to get a Hermitian group. 


\section{A remark}

Let $\mathbb{Z}$ act on $\mathbb{Q}_p$ by multiplication by $p$ and form the corresponding semi-direct product $\mathbb{Q}_p \rtimes \mathbb{Z}$. It is posed as an open question in \cite[Section 3.6]{Pal15} and a recent article by the author \cite[Section 8.5]{Car25}, to determine whether the group $\mathbb{Q}_p \rtimes \mathbb{Z}$, and certain generalisations of this group, are Hermitian. 

By a similar argument to the last section, it can be checked that the induced action of $\mathbb{Z}$ on $\widehat{\mathbb{Q}_p}$ has infinitely many orbits. The same is the case for the various generalisations of this group stated in \cite{Pal15} and \cite{Car25}. In particular, the example of the last section shows that we cannot expect to be able to generalise Corollary \ref{cor:main} to prove Hermitianness of the group $\mathbb{Q}_p \rtimes \mathbb{Z}$ and its generalisations. New results will need to be used/developed to understand the Hermitian property of these semi-direct products.

If these groups are proved to be Hermitian, then this will provide further examples of non-discrete totally disconnected locally compact Hermitian groups with exponential growth.


\bibliographystyle{amsplain}
\bibliography{aff_her}


\end{document}